\documentclass[12pt]{article}%
\usepackage{amssymb}%
\usepackage{amsmath}%
\usepackage{amsfonts}%
\usepackage{graphicx}
\providecommand{\U}[1]{\protect\rule{.1in}{.1in}}
\newtheorem{theorem}{Theorem}

\newenvironment{proof}[1][Proof]{\noindent\textbf{#1.} }{\ \rule{0.5em}{0.5em}}
\begin{document}

\title{Walker's cancellation theorem}
\author{Robert Lubarsky and Fred Richman\\Florida Atlantic University}
\date{22 August 2012}
\maketitle

\begin{abstract}
Walker's cancellation theorem says that if $B\oplus\mathbf{Z}$ is isomorphic
to $C\oplus\mathbf{Z}$ in the category of abelian groups, then $B $ is
isomorphic to $C$. We construct an example in a diagram category of abelian
groups where the theorem fails. As a consequence, the original theorem does
not have a constructive proof even if $B$ and $C$ are subgroups of the free
abelian group on two generators. Both of these results contrast with a group
whose endomorphism ring has stable range one, which allows a constructive
proof of cancellation and also a proof in any diagram category.

\end{abstract}

\section{Cancellation}

An object $G$ in an additive category is \textbf{cancellable} if whenever
$B\oplus G$ is isomorphic to $C\oplus G$, then $B$ is isomorphic to $C$.
Elbert Walker, in his dissertation \cite{W}, and P. M. Cohn in \cite{C},
independently answered a question of Irving Kaplansky by showing that finitely
generated abelian groups are cancellable in the category of abelian groups.
The most interesting case is that of $\mathbf{Z}$, the additive group of
integers. That's because finitely generated groups are direct sums of copies
of $\mathbf{Z}$ and of cyclic groups of prime power order, and a cyclic group
of prime power order has a local endomorphism ring, hence is cancellable by a
theorem of Azumaya \cite{Az}.

It is somewhat anomalous that $\mathbf{Z}$ is cancellable. A rank-one
torsion-free group $A$ is cancellable if and only if $A\cong\mathbf{Z}$ or the
endomorphism ring of $A$ has stable range one \cite[Theorem 8.12]{A}%
,\cite{FL}. (A ring $R$ has\textbf{\ stable range one} if whenever $aR+bR=R$,
then $a+bR$ contains a unit of $R$.) Thus for rank-one torsion-free groups,
the endomorphism ring tells the whole story---except for $\mathbf{Z}$. It
turns out that an object is cancellable if its endomorphism ring has stable
range one. The proof of this in \cite[Theorem 4.4]{L} is constructive and
works for any abelian category. It is also true, \cite{L}, that semilocal
rings have stable range one, so Azumaya's theorem is a special case of this.
In fact, that the endomorphism ring of $A$ has stable range one is equivalent
to $A$ being substitutable, a stronger condition than cancellation
\cite[Theorem 4.4]{L}. We say that $A$ is \textbf{substitutable} if any two
summands a group, with complements that are isomorphic to $A$, have a common
complement. The group $\mathbf{Z}$ is not substitutable: Consider the
subgroups of $\mathbf{Z}^{2}$ generated by $\left(  1,0\right)  $, $\left(
0,1\right)  $, $\left(  7,3\right)  $, and $\left(  5,2\right)  $. The first
and second, and the third and fourth, are complementary summands. The second
and fourth do not have a common complement because that would require $\left(
a,b\right)  $ with $a=\pm1$ and $2a-5b=\pm1$.

In this paper we will investigate whether $\mathbf{Z}$ is cancellable in the
(abelian) category $\mathcal{D}_{T}\left(  \mathbf{Ab}\right)  $ of diagrams
of abelian groups based on a fixed finite poset $T$ with a least element.
There is a natural embedding of $\mathbf{Ab}$ into $\mathcal{D}_{T}\left(
\mathbf{Ab}\right)  $ given by taking a group into the constant diagram on $T$
with identity maps between the groups on the nodes. In particular, we can
identify the group of integers as an object of $\mathcal{D}_{T}\left(
\mathbf{Ab}\right)  $. As the endomorphism ring of any group $G$ is the same
as that of its avatar in $\mathcal{D}_{T}\left(  \mathbf{Ab}\right)  $, a
substitutable group is substitutable viewed as an object in $\mathcal{D}%
_{T}\left(  \mathbf{Ab}\right)  $. However it turns out that $\mathbf{Z}$ is
not cancellable in $\mathcal{D}_{T}\left(  \mathbf{Ab}\right)  $ where $T$ is
the linearly ordered set $\left\{  0,1,2\right\}  $.

This result has repercussions for the constructive theory of abelian groups.
Because of it, we can conclude that Walker's theorem does not admit a
constructive proof. In fact, it is not even provable when $B$ and $C$ are
restricted to be subgroups of $\mathbf{Z}^{2}$. It was the question of whether
Walker's theorem had a constructive proof that initiated our investigation.
You can think of a constructive proof as being a proof within the context of
intuitionistic logic. Such proofs are normally constructive in the usual
informal sense. Most any proof of Azumaya's theorem is constructive, so a
constructive proof of the cancellability of $\mathbf{Z}$ would show that you
can cancel finite direct sums of finite and infinite cyclic groups.

As any homomorphism from an abelian group onto $\mathbf{Z}$ splits, Walker's
theorem can be phrased as follows: If $A$ is an abelian group, and
$f,g:A\rightarrow\mathbf{Z}$ are epimorphisms, then $\ker f\cong\ker g$. The
following theorem gets us part way to a proof of Walker's theorem.

\begin{theorem}
\label{proof}Let $A$ be an abelian group and $f,g:A\rightarrow\mathbf{Z}$ be
epimorphisms. Then $f\left(  \ker g\right)  =g\left(  \ker f\right)  $ so that%
\[
\frac{\ker g}{\ker f\cap\ker g}\cong f\left(  \ker g\right)  =g\left(  \ker
f\right)  \cong\frac{\ker f}{\ker f\cap\ker g}
\]

\end{theorem}

\begin{proof}
Consider the image $I$ of the map $A\rightarrow\mathbf{Z}\oplus\mathbf{Z}$
induced by $f$ and $g$. As $f$ and $g$ are epimorphisms, $I$ is a subdirect
product. Note that $f\left(  \ker g\right)  =I\cap\left(  \mathbf{Z}%
\oplus0\right)  $ when the latter is viewed as a subgroup of $\mathbf{Z}$, and
similarly $g\left(  \ker f\right)  =I\cap\left(  0\oplus\mathbf{Z}\right)  $.
To finish the proof we show that if $\left(  x,0\right)  \in I$, then $\left(
0,x\right)  \in I$. As $I$ is a subdirect product, there exists $n\in
\mathbf{Z}$ such that $\left(  n,1\right)  \in I$. Thus $\left(  0,x\right)
=x\left(  n,1\right)  -n\left(  x,0\right)  \in I$.\medskip
\end{proof}

Thus we get the desired isomorphism $\ker f\cong\ker g$ if $\ker f\cap\ker
g=0$ or if $f\left(  \ker g\right)  $ is projective. Classically, every
subgroup of $\mathbf{Z}$ is projective, so this constitutes a classical proof.
Indeed, it is a classical proof that in the category of modules over a
Dedekind domain $D$, the module $D$ is cancellable \cite{H}.

\section{The example}

Our example lives in the category $\mathcal{D}_{T}\left(  \mathbf{Ab}\right)
$ of diagrams of abelian groups based on the linearly ordered set $T=\left\{
0,1,2\right\}  $. The example shows that you can't cancel $\mathbf{Z}$ in
$\mathcal{D}_{T}\left(  \mathbf{Ab}\right)  $.

The groups on the nodes will be subgroups $A_{0}\subset A_{1}\subset
A_{2}=\mathbf{Z}^{3}$ defined by generators:%
\[
A_{0}=%
\begin{array}
[c]{c}%
\left(  1,3,0\right) \\
\left(  3,1,0\right)
\end{array}
\ A_{1}=%
\begin{array}
[c]{c}%
\left(  1,0,-24\right) \\
\left(  0,1,8\right) \\
\left(  0,0,64\right)
\end{array}
\ A_{2}=%
\begin{array}
[c]{c}%
\left(  1,0,0\right) \\
\left(  0,1,0\right) \\
\left(  0,0,1\right)
\end{array}
\]
Note that $\left(  0,8,0\right)  ,\left(  8,0,0\right)  \in A_{0}$. The maps
between these groups are inclusions. Define the maps $f,g:\mathbf{Z}%
^{3}\rightarrow\mathbf{Z}$ by $f\left(  a,b,c\right)  =a$ and $g\left(
a,b,c\right)  =b$. The maps $f$ and $g$ each induce maps from these three
groups into $\mathbf{Z}$ which give two maps from the diagram into the
constant diagram $\mathbf{Z}$. We denote the kernel of the map $f$ restricted
to $A_{i}$ by $\ker_{i}f$ and similarly for $g$. These kernels admit the
following generators:%
\[
\ker_{0}f=\left(  0,8,0\right)  \ \ \ker_{1}f=%
\begin{array}
[c]{c}%
\left(  0,1,8\right) \\
\left(  0,0,64\right)
\end{array}
\ \ \ker_{2}f=%
\begin{array}
[c]{c}%
\left(  0,1,0\right) \\
\left(  0,0,1\right)
\end{array}
\]%
\[
\ker_{0}g=\left(  8,0,0\right)  \ \ \ker_{1}g=%
\begin{array}
[c]{c}%
\left(  1,0,-24\right) \\
\left(  0,0,64\right)
\end{array}
\ \ \ker_{2}g=%
\begin{array}
[c]{c}%
\left(  1,0,0\right) \\
\left(  0,0,1\right)
\end{array}
\]

The diagrams $B=\ker f$ and $C=\ker g$ are clearly each embeddable in the
diagram $\mathbf{Z}\oplus\mathbf{Z}$. That $B\oplus\mathbf{Z}$ is isomorphic
to $C\oplus\mathbf{Z}$ follows from the fact that the diagram $A$ can be
written as an internal direct sum $B\oplus\mathbf{Z}$ and also as an internal
direct sum $C\oplus\mathbf{Z}$. The generator of $\mathbf{Z}$ in the first
case is the element $\left(  1,3,0\right)  $, in the second case $\left(
3,1,0\right)  $.

\begin{theorem}
\label{Theorem no iso}There is no isomorphism between $\ker f$ and $\ker g$ in
$\mathcal{D}_{T}\left(  \mathbf{Ab}\right)  $.
\end{theorem}

\begin{proof}
Suppose we had an isomorphism $\varphi:\ker f\rightarrow\ker g$. Looking at
the isomorphisms at $0$ and $2$, there exist $e,e^{\prime}=\pm1$ and
$x\in\mathbf{Z}$ so that
\[
\varphi\left(  0,8,0\right)  =\left(  8e,0,0\right)  \text{ and }%
\varphi\left(  0,0,1\right)  =\left(  x,0,e^{\prime}\right)
\]
Thus $\varphi\left(  0,1,8\right)  =\left(  e+8x,0,8e^{\prime}\right)  $. For
$\left(  e+8x,0,8e^{\prime}\right)  $ to be in $\ker_{1}g$, we must have
$8e^{\prime}+24\left(  e+8x\right)  $ divisible by $64$. But $8e^{\prime
}+24\left(  e+8x\right)  $ is equal to $8e^{\prime}+24e$ modulo $64$, and this
is not divisible by $64$.\medskip
\end{proof}

The following result shows that we can't get an example that is a subobject of
the diagram $\mathbf{Z}^{n}$ using the linearly ordered set $T=\left\{
0,1\right\}  $.

\begin{theorem}
Let $T=\left\{  0,1\right\}  $. In the category $\mathcal{D}_{T}\left(
\mathbf{Ab}\right)  $, if $A$ and $B$ are subobjects of $\mathbf{Z}^{n}$, and
$A\oplus\mathbf{Z}$ is isomorphic to $B\oplus\mathbf{Z}$, then $A$ is
isomorphic to $B$.
\end{theorem}

\begin{proof}
Write $A\subseteq\mathbf{Z}^{n}$ as $A_{0}\subseteq A_{1}$. As $A_{1}$ is a
finite-rank free abelian group, the situation $A_{0}\subseteq A_{1}$ can be
represented by an integer matrix whose rows generate $A_{0}$. Using elementary
row and column operations, we can diagonalize this matrix so that each entry
on the diagonal divides the next (Smith normal form). Thus $A$ is isomorphic
to $B$ exactly when the ranks of the free abelian groups $A_{1}$ and $B_{1}$
are equal, and $A_{1}/A_{0}\cong B_{1}/B_{0}$. If $C=A\oplus\mathbf{Z}$ is
isomorphic to $D=B\oplus\mathbf{Z}$, then the rank of $C_{1}=A_{1}%
\oplus\mathbf{Z}$ is equal to the rank of $D_{1}=B_{1}\oplus\mathbf{Z}$, so
the rank of $A_{1}$ is equal to the rank of $B_{1}$, and $A_{1}/A_{0}\cong
C_{1}/C_{0}\cong D_{1}/D_{0}\cong B_{1}/B_{0}$, so $A$ is isomorphic to $B$.
\medskip
\end{proof}

This theorem leaves open the question of whether there is an counterexample of
this sort using the poset that looks like a \textquotedblleft
V\textquotedblright.

\section{The Brouwerian counterexample}

A Brouwerian example is an object depending on a finite family of
propositions. The idea is that if a certain statement holds about that object,
then some relation holds among the propositions. Thus a Brouwerian example is
piece of \textit{reverse mathematics}: the derivation of a propositional
formula from a mathematical statement. For example, there may be just one
proposition $P$ and if the statement holds for that object, then $P\vee\lnot
P$ holds. Thus from the general truth of the statement we could derive the law
of excluded middle, from which we would conclude that the statement does not
admit a constructive proof. Our Brouwerian counterexample to Walker's theorem
is based on the diagram of groups of the previous section.

Let $P$ and $Q$ be propositions. Let
\[
A=\left\{  x\in\mathbf{Z}^{3}:x\in A_{0}\text{ or }P\wedge x\in A_{1}\text{ or
}P\wedge Q\right\}
\]
where $A_{0}$ and $A_{1}$ are defined in the preceding section. The maps
$f,g:\mathbf{Z}^{3}\rightarrow\mathbf{Z}$ are defined as before by $f\left(
a,b,c\right)  =a$ and $g\left(  a,b,c\right)  =b$.

Note that $A$ is a discrete group (any two elements are either equal or
distinct) as it is a subgroup of the discrete group $\mathbf{Z}^{3}$.

\begin{theorem}
The groups $\ker f$ and $\ker g$ are isomorphic if and only if $P\vee
P\Rightarrow\left(  Q\vee\lnot Q\right)  $.
\end{theorem}

\begin{proof}
As before, we denote $A_{i}\cap\ker f$ by $\ker_{i}f$.

If $P$ holds, then the isomorphism is induced by $\varphi\left(  0,1,0\right)
=\left(  1,0,-32\right)  $ and $\varphi\left(  0,0,1\right)  =\left(
0,0,1\right)  $. Suppose $P\Rightarrow\left(  Q\vee\lnot Q\right)  $ holds.
Define $\varphi$ on $\ker_{0}f$ by $\varphi\left(  0,8,0\right)  =\left(
0,0,8\right)  $. That's all we have to do unless we are given $x$ that is not
in $\ker_{0}f$. If $x\in\ker_{2}f$, and $x\notin\ker_{0}f$, then $P$ holds,
hence either $Q$ or $\lnot Q$ holds. If $Q$ holds, then the isomorphism is
induced by $\varphi\left(  0,1,0\right)  =\left(  1,0,0\right)  $ and
$\varphi\left(  0,0,1\right)  =\left(  0,0,1\right)  $. If $\lnot Q$ holds,
the isomorphism is induced by $\varphi\left(  0,1,8\right)  =\left(
3,0,-8\right)  $ and $\varphi\left(  0,8,0\right)  =\left(  8,0,0\right)  $.

Conversely, suppose $\varphi$ is an isomorphism. If $\varphi\left(
0,8,0\right)  \neq\left(  \pm8,0,0\right)  $, then $P$ holds, so we may assume
that $\varphi\left(  0,8,0\right)  =\left(  8,0,0\right)  $. To show that
$P\Rightarrow Q\vee\lnot Q$, suppose $P$ holds. If $\varphi\left(  \ker
_{1}f\right)  \neq\ker_{1}g$, then $Q$ holds. If $\varphi\left(  \ker
_{1}f\right)  =\ker_{1}g$, then $Q$ cannot hold because that would give an
isomorphism in the diagram category contrary to Theorem \ref{Theorem no iso}.
\medskip
\end{proof}

So if we could find a constructive proof that $\ker f$ and $\ker g$ were
isomorphic, then we would have a constructive proof of the propositional form
\[
P\vee P\Rightarrow\left(  Q\vee\lnot Q\right)  \text{.}
\]
That means that this form would be a theorem in the intuitionistic
propositional calculus. But then by the disjunction property, either $P$ is a
theorem, which it is not, or $P\Rightarrow\left(  Q\vee\lnot Q\right)  $ is a
theorem. In the latter case, substituting $\top$ for $P$ gives $Q\vee\lnot Q$,
the law of excluded middle, which is not a theorem.

The diagram example of the preceding section can itself be thought of as an
object in a model of intuitionistic abelian group theory, and in this way
directly shows that Walker's theorem does not admit a constructive proof, even
for subgroups of $\mathbf{Z}^{2}$.

\section{Canceling $\mathbf{Z}$ with respect to subgroups of $\mathbf{Q}$}

We have seen that we can't cancel $\mathbf{Z}$ with respect to certain
subgroups of $\mathbf{Z}\oplus\mathbf{Z}$. It is natural to ask what the
situation is with respect to subgroups of $\mathbf{Z}$. We give a constructive
proof of the following theorem.

\begin{theorem}
Let $B$ be an abelian group such that every nontrivial homomorphism from $B$
to $\mathbf{Z}$ is one-to-one. If $f$ is a homomorphism from $B\oplus
\mathbf{Z}$ onto $\mathbf{Z}$, then $\ker f$ is isomorphic to $mB$ for some
positive integer $m$. Hence if $B$ is torsion free, then $\ker f$ is
isomorphic to $B$.
\end{theorem}

\begin{proof}
Let $s=f\left(  0,1\right)  $ and $f_{1}$ the restriction of $f$ to $B$. As
$f$ is onto, we have $f_{1}\left(  B\right)  +s\mathbf{Z}=\mathbf{Z}$. If
$s=0$, then $f$ maps $B$ isomorphically onto $\mathbf{Z}$, and $0\oplus
\mathbf{Z}$ is $\ker f$, in which case we can set $m=1$. So we may suppose
that $s>0$. We will show that $\ker f$ is isomorphic to $sB$.

When is $\left(  b,n\right)  $ in $\ker f$? As $f\left(  b,n\right)
=f_{1}\left(  b\right)  +sn$, we see that a necessary and sufficient condition
is that $f_{1}\left(  b\right)  \in s\mathbf{Z}$ and $n=-f_{1}\left(
b\right)  /s$. Thus $\ker f$ is isomorphic to $f_{1}^{-1}\left(
s\mathbf{Z}\right)  $. As $f_{1}\left(  B\right)  +s\mathbf{Z}=\mathbf{Z}$,
and $f_{1}\left(  B\right)  $ and $s\mathbf{Z}$ are ideals of $\mathbf{Z}$, it
follows that $f_{1}\left(  B\right)  \cap s\mathbf{Z}=f_{1}\left(  B\right)
s\mathbf{Z}=sf_{1}\left(  B\right)  $. Thus
\[
f_{1}^{-1}\left(  s\mathbf{Z}\right)  =f_{1}^{-1}\left(  f_{1}\left(
B\right)  \cap s\mathbf{Z}\right)  =f_{1}^{-1}\left(  sf_{1}\left(  B\right)
\right)
\]
Clearly $f_{1}^{-1}\left(  sf_{1}\left(  B\right)  \right)  \supseteq sB$.
Conversely, if $f_{1}\left(  b\right)  \in sf_{1}\left(  B\right)
=f_{1}\left(  sB\right)  $, then $f_{1}\left(  b\right)  =f_{1}\left(
sb^{\prime}\right)  $ so $b=sb^{\prime}\in sB$.\medskip
\end{proof}

Note that any torsion-free group $B$ of rank at most one satisfies the
hypothesis of the theorem. Also any group with no nontrival maps into
$\mathbf{Z}$. Classically, this latter condition simply says that $B$ has no
proper $\mathbf{Z}$ summands.

What other groups $B$ allow cancellation of $\mathbf{Z}$? It suffices that $B
$ be finitely generated. To see this, look at Theorem \ref{proof}. If $\ker f
$ is finitely generated, then $g\left(  \ker f\right)  $ is a finitely
generated subgroup of $\mathbf{Z}$, hence is projective. From this argument it
suffices that any image of $B$ in $\mathbf{Z}$ be finitely generated. Notice
that subgroups of $\mathbf{Z}$ need not have this property.

What about a direct sum of two groups that allow cancellation of $\mathbf{Z}$,
such as a direct sum of two subgroups of $\mathbf{Z}$?

\end{document}